\documentclass[11pt]{amsart}
\usepackage{}
\usepackage{cases}
\usepackage{graphicx}
\usepackage{amsmath}
\usepackage{amssymb}
\usepackage{amsfonts}
\usepackage[mathscr]{eucal}
\makeatletter
\def\numberwithin#1#2{\@ifundefined{c@#1}{\@nocnterrr}{
  \@ifundefined{c@#2}{\@nocnterr}{
  \@addtoreset{#1}{#2}
  \toks@\expandafter\expandafter\expandafter{\csname the#1\endcsname}
  \expandafter\xdef\csname the#1\endcsname
    {\expandafter\noexpand\csname the#2\endcsname
     .\the\toks@}}}}
\makeatother
\numberwithin{equation}{section}

\newtheorem{theorem}{Theorem}[section]

\newtheorem{lemma}[theorem]{Lemma}

\newtheorem*{problem*}{Problem 1}

\newtheorem{remark}[theorem]{Remark}

\newtheorem{thm}{Theorem}

\begin{document}
\title[Monotonicity of functionals for conformal Ricci flow]
{Monotonicity of functionals along conformal Ricci flow}
\thanks{$^\ast$ The corresponding author, email: jianhongwang@163.com}

\author{Fengjiang Li}
\address{School of Mathematical Sciences, East China Normal University,
500 Dongchuan Road, Shanghai 200241,  P. R. of China}
\email{lianyisky@163.com}

\author{Peng Lu}
\address{Department of Mathematics, University of Oregon, Eugene, OR 97403, USA}
\email{penglu@uoregon.edu}

\author{Jianhong Wang$^\ast$}
\address{School of Mathematical Sciences, East China Normal University,
500 Dongchuan Road, Shanghai 200241,  P. R. of China}
\email{jianhongwang@163.com}

\author{Yu Zheng}
\address{School of Mathematical Sciences, East China Normal University,
500 Dongchuan Road, Shanghai 200241,  P. R. of China}
\email{zhyu@math.ecnu.edu.cn}

\subjclass[2010]{53C44}
\date{\today}

\keywords{Conformal Ricci flow, conjugate heat equation, monotonicity of functionals}

\begin{abstract}
The main purpose of this note is to construct two functionals of the positive solutions to
the conjugate heat equation associated to the metrics evolving
by the conformal Ricci flow on closed manifolds.
We show that they are nondecreasing by calculating the explicit evolution
formulas of these functionals.
For the entropy functional we give another proof of the monotonicity by establishing
a pointwise formula.
Moreover, we show that the increase are strict unless the metrics are Einstein.
\end{abstract}

\maketitle

\section{Introduction}

Monotonicity formulas (geometric quantities which are monotone along geometric flows)
play a key role in geometry and analysis, when being used deliberately they give various kind
geometric information including rigidity results.
Two outstanding examples are Huisken's monotonicity formula for the mean curvature flow \cite{Hui90}
and Perelman's $\mathcal{W}$-entropy for the Ricci flow \cite{Per02}.
Inspired by the work \cite{EKNT08} and \cite{Mul10}, a general approach to find monotone quantities
on manifolds with evolving metrics is developed by H.-X. Guo, R. Philipowski, and A. Thalmaier
in \cite{GPT13}.

The conformal Ricci flow (CRF in short)  is proposed by A. Fischer \cite{Fis04}
and is further studied in, for example, \cite{LQZ14}, \cite{SZ18},  and \cite{LQZ19}.
Going into a little details CRF considers a family of metrics $g=g(t), t\in[0,T)$, on an $(m+1)$-dimensional
closed manifold $M$, which is evolved by
\begin{equation}\label{eq crf 1}
\left\{
\begin{aligned}
& \partial_t g=-2\left ( \operatorname{Rc}_{g(t)}+(p(t)+m)g(t)\right ) ~~~~
        &\text{on}~~~&M\times [0,T),&\\
& \left ( -\Delta_{g(t)}+(m+1) \right )p(t)=\frac{1}{m}|\operatorname{Rc}_{g(t)}+mg(t)|^2
 ~~~~~  &\text{on}~~~~&M\times[0,T),& \\
&g(0)=g_{0},&
\end{aligned}     \right.
\end{equation}
where $\operatorname{Rc}$ is the Ricci curvature, $p(t)$ is a scalar function,
 and $T$ is a positive constant.
We assume that the initial metric $g_0$ has constant scalar curvature $-m(m+1)$.

In this note, we apply the approach in \cite{GPT13} to CRF
and find two new monotone quantities $\tilde{\mathcal{E}}(g,u)$  (see (\ref{eq functional mathcal E})
for its definition and Lemma \ref{thm E functional monotone})
and entropy $\tilde{\mathcal{W}}(g,u)$ defined by
\begin{equation} \label{eq entropy W for CRF}
\tilde{\mathcal{W}}(g,u)=\int_{M}\big(|\nabla \ln u|^2+2(m+1) \ln u\big)ud{\mu_g},
\end{equation}
where $u$ is a positive function.
More precisely we will consider the conjugate heat equation associated to CRF
\begin{equation}\label{eq 2 old}
\partial_ t u=-\Delta_{g(t)} u+(m+1)p(t)u, \quad t \in [0,T).
\end{equation}
It is easy to check that $\frac{d}{dt}\int_M u (t) d\mu_{g(t)} =0$.
We will derive the following formula which is the main result of this note.

\begin{theorem} \label{thm funct W monotone}
Let $g(t)$ be a solution of CRF (\ref{eq crf 1}) and let $u(x,t)$
be a positive solution of the conjugate heat equation (\ref{eq 2 old}).
Define $\tilde{\mathcal{W}}(t) =\tilde{\mathcal{W}}(g(t), u(t))$.
Then we have

\begin{align}
\frac{d \tilde{\mathcal{W}}(t)}{dt}=&2 \int_M \left |\operatorname{Rc}
 +mg-\nabla\nabla \ln u \right |^2u d\mu_{g(t)}
+\frac{2}{m}\int_M|\operatorname{Rc}+mg|^2ud\mu_{g(t)}  \notag \\
&+2\int_M|\nabla \ln u|^2pu d \mu_{g(t)}+2\int_M|\nabla \ln u|^2u d \mu_{g(t)}
\label{eq 14 old}
\end{align}
for $ t \in [0,T)$. Hence $\tilde{\mathcal{W}}(t)$ is nondecreasing, and the increase is strict
unless $g(t)$ is an Einstein metric.
\end{theorem}

We will give two proofs of this theorem in \S \ref{sec proof of main result} and \S \ref{sec second proof},
respectively. In the first proof we will see how the definition of $\tilde{\mathcal{W}}$ in
(\ref{eq entropy W for CRF}) comes up. The second proof is a mimic of the proof of Perelman's
entropy monotonicity given in \cite[\S 9.1]{Per02}.

It should be pointed out that there is a long list of works on monotonicity formulas for various
geometric flow (including static manifolds), many of them are inspired by the work in \cite{Hui90}
and \cite{Per02}.
We mention a few here as examples,
L. Ni's entropy on static Riemannian manifolds \cite{Ni04},
List's entropy for extended Ricci flow \cite{Lis08},
M\"{u}ller's entropy for Ricci-Harmonic flow \cite{Mul12},
and L.-F. Wang's entropy for Ricci-Bourguignon flow \cite{Wa19}.
The fancy part of Theorem \ref{thm funct W monotone} is that it is a monotonicity formula
for coupled parabolic and elliptic equations.

\section{Proof of Theorem \ref{thm funct W monotone}} \label{sec proof of main result}

To facilitate our discussion of the derivation of  (\ref{eq entropy W for CRF}),
we now give a brief review of the general approach in \cite{GPT13}.
 Let $\hat{M}^n$ be an n-dimensional closed manifold.
We consider an abstract geometric flow of metrics $g(t)$ on $M$ defined by
 \begin{equation} \label{eq abstract flow gu}
\partial_t g=-2\alpha, \quad t \in [0,T),
\end{equation}
where $\alpha$ is a time-dependent symmetric two-tensor on $\hat{M}$.
Let $A =\operatorname{tr}_g \alpha$.
H.-X. Guo et al. construct the analog of Perelman's $\mathcal{F}$-energy and
$\mathcal{W}$-entropy by defining
\begin{equation} \label{eq abstract flow energy gu}
 \mathcal{F}(g,\phi)=\int_{\hat{M}}(A+|\nabla \phi|^{2})e^{-\phi}d\mu_g
 \end{equation}
for some function $\phi$, and
\begin{equation} \label{eq abstract flow entropy gu}
\mathcal{W}(g,\psi,\tau)=\int_{\hat{M}}\big(\tau(A+|\nabla \psi|^{2})+\psi-n\big)
(4\pi\tau)^{-\frac{n}{2}}e^{-\psi}d\mu_g
\end{equation}
for $\tau >0$ and function $\psi$ satisfying
$\int_{\hat{M}} (4\pi\tau)^{-\frac{n}{2}}e^{-\psi}d\mu_g =1$, respectively.
They prove the following

\begin{thm}\label{T1.1} \text{(\cite[Theorem 4.2 and 5.2]{GPT13})}
(i) Suppose that $\phi$ satisfies the backward heat equation
\begin{equation} \label{eq phi backward energy}
\partial_t \phi =-\Delta \phi+|\nabla \phi|^{2}-A, \quad t \in [0,T).
\end{equation}
Then under the flow (\ref{eq abstract flow gu})  the energy $\mathcal{F}(t)
=\mathcal{F}(g(t), \phi(t))$ satisfies
$$ \frac{d \mathcal{F}(t)}{dt}=2\int_{\hat{M}}\big(|\alpha+\nabla\nabla \phi|^{2}+
\hat{\Theta}(-\nabla \phi)\big)e^{-\phi}d\mu_g,$$
where
$$\hat{\Theta}(V)=( \operatorname{Rc}-\alpha)(V,V)+\langle\nabla A-2
\operatorname{Div}(\alpha),V\rangle+
\frac{1}{2}( \partial_t A-2|\alpha|^2-\Delta A)$$
for any vector field $V$.
In particular, $\mathcal{F}(t)$ is nondecreasing
when $\hat{\Theta} \geq 0$, and the increase is strict unless
$$\alpha+\nabla\nabla \phi=0 \quad \text{and} \quad
\hat{\Theta} (-\nabla \phi)=0.$$

\noindent (ii) Suppose that $\psi$ satisfies the backward heat equation
\begin{equation} \label{eq psi backward entropy}
\partial_t \psi=-\Delta \psi+|\nabla \psi|^{2}-A+\frac{n}{2\tau},
\quad t \in [0,T),
\end{equation}
where $\tau=T-t$.
Then under the flow (\ref{eq abstract flow gu})  the entropy $\mathcal{W}(t)
=\mathcal{W}(g(t), \psi(t),$ $ T-t)$ satisfies
$$ \frac{d \mathcal{W}(t)}{dt}=\int_{\hat{M}}2\tau\big(|\alpha+\nabla\nabla \psi-
\frac{1}{2\tau}g|^{2}+ \hat{\Theta} (-\nabla \psi)\big)
(4\pi\tau)^{-\frac{n}{2}}e^{-\psi}d\mu_g.$$
In particular,  $\mathcal{W}(t)$ is  nondecreasing
when $\hat{\Theta} \geq 0$,
and the increase is strict unless
$$\alpha+\nabla\nabla \psi-\frac{1}{2\tau}g=0 \quad \text{and} \quad
\hat{\Theta}(-\nabla \psi)=0.$$
\end{thm}

\vskip .2cm
On closed Riemannian manifold $(M^{m+1},g)$, we define the
Boltzmann-Shannon entropy $\tilde{\mathcal{E}}(g,u)$ by
\begin{equation} \label{eq functional mathcal E}
\tilde{\mathcal{E}}(g,u)=\int_M u\ln ud \mu_{g},
\end{equation}
where $u$ is a positive function.
Towards the proof of Theorem \ref{thm funct W monotone} we start with the following

\begin{lemma} \label{thm E functional monotone}
Let $g(t)$ be a solution of CRF  (\ref{eq crf 1}) and let $u(x,t)$ be a positive solution
to the conjugate heat equation (\ref{eq 2 old}).
Denote $\tilde{\mathcal{E}}(t) = \tilde{\mathcal{E}}(g(t),u(t))$, then we have
\begin{equation}\label{eq 3 old}
\frac{d \tilde{\mathcal{E}}(t)}{dt}=\int_M\big(|\nabla \ln u|^2+(m+1)p\big)ud \mu_g.
\end{equation}
Thus $\tilde{\mathcal{E}}(t)$ is nondecreasing and the increase is strict
unless $g(t)$ is an Einstein metric.
\end{lemma}

\begin{proof}
By a direct computation or using \cite[Theorem 2.1]{GPT13} we have
\begin{align*}
\frac{d \tilde{\mathcal{E}}(t)}{dt}&=\int_M \left (\partial_ t u \cdot \ln u
+ \partial_t u-(m+1)pu\ln u \right )d \mu_{g(t)}\\
&=\int_M \big(-\Delta u\ln u+(m+1)pu\big)d \mu_{g(t)}\\
&=\int_M\big(|\nabla\ln u|^2+(m+1)p\big)ud \mu_{g(t)},
\end{align*}
where we have used the integration by parts to get the last inequality.

It is known that $p\geq 0$ \cite{Fis04}, hence $\frac{d \tilde{\mathcal{E}}(t)}{dt} \geq 0$.
When $\frac{d \tilde{\mathcal{E}}(t)}{dt}=0$, then $p(t)=0$,
by equation (\ref{eq crf 1}) we have $\operatorname{Rc}_{g(t)}=-mg(t)$.
\end{proof}

In the following we will not directly use Theorem A to prove Theorem \ref{thm funct W monotone},
rather we use the method in \cite{GPT13} to derive the correct formula for functional $\tilde{\mathcal{W}}$
as given in (\ref{eq entropy W for CRF})
which is not any of the functionals defined in (\ref{eq abstract flow energy gu}) and
(\ref{eq abstract flow entropy gu}) when $\alpha$ is defined by CRF (\ref{eq crf 1}).

\vskip.2cm
\noindent \emph{Proof} of Theorem \ref{thm funct W monotone}.\
When the abstract flow (\ref{eq abstract flow gu}) is CRF  (\ref{eq crf 1}), we have
$A=(m+1)p$. It follows from \cite[Theorem 2.1]{GPT13} that
\begin{equation}\label{eq 5 old}
\frac{d^2 \tilde{\mathcal{E}}(t)}{dt^2} =2\int_M\left ( |\operatorname{Rc}+(p+m)g
-\nabla\nabla\ln u|^2+ \Theta( \nabla \ln u)\right ) ud \mu_g,
\end{equation}
where
\begin{align*}
\Theta(\nabla\ln u)=& -(p+m)|\nabla\ln u|^2+(m-1)\langle\nabla p,\nabla\ln u\rangle \\
& +\frac{m+1}{2}(\partial_t p-\Delta p)-|\operatorname{Rc} +(p+m)g|^2.
\end{align*}

First we rewrite the following four integrals appeared in (\ref{eq 5 old}).
By integration by parts we have
\begin{align}
\int_M |\nabla\ln u|^2pud \mu_{g(t)} &=\int_M \langle\nabla\ln u, \nabla u\rangle pd \mu_{g(t)} \notag \\
&=-\int_M\langle\nabla p, \nabla u\rangle d \mu_{g(t)} -\int_M\Delta\ln u\cdot pud \mu_{g(t)} \notag \\
&=\int_M\Delta p\cdot ud \mu_{g(t)}-\int_M\Delta\ln u\cdot pud \mu_{g(t)}. \label{eq 7 old}
\end{align}
By (\ref{eq 3 old}) we have
\begin{equation} \label{eq 5 new add name}
\int_M |\nabla\ln u|^2ud \mu_{g(t)}=\frac{d \tilde{\mathcal{E}}(t)}{dt}-(m+1)\int_M pud \mu_{g(t)}.
\end{equation}
We also have
\begin{equation} \label{eq 5 new new add name}
\int_M\langle\nabla p,\nabla \ln u\rangle ud \mu_{g(t)}=-\int_M\Delta p\cdot ud \mu_{g(t)}.
\end{equation}
Finally using equation (\ref{eq 2 old}) and the integration by parts again we have
\begin{align}
\int_M \partial_t p \cdot ud \mu_{g(t)}&=\frac{d}{dt}\int_M pud\mu_{g(t)}
-\int_M p \cdot \partial_t ud \mu_{g(t)}+(m+1)\int_M p^2ud \mu_{g(t)} \notag \\
&=\frac{d}{dt}\int_M pud\mu_{g(t)}+\int_M p\Delta ud \mu_{g(t)} \notag \\
&=\frac{d}{dt}\int_M pud \mu_{g(t)}+\int_M u\Delta pd \mu_{g(t)}. \label{eq 8 old}
\end{align}
Plugging (\ref{eq 7 old}), (\ref{eq 5 new add name}), (\ref{eq 5 new new add name}),
 and (\ref{eq 8 old}) into (\ref{eq 5 old}), and rearranging terms we get
\begin{align}
&\frac{d}{dt}\int_{M}\big(|\nabla\ln u|^2+2m\ln u\big)ud \mu_{g(t)}  \notag \\
=& \frac{d^2 \tilde{\mathcal{E}}(t)}{dt^2}+2m \frac{d \tilde{\mathcal{E}}(t)}{dt}
-(m+1)\frac{d}{dt}\int_M pud \mu_{g(t)} \notag \\
= & 2\int_M| \operatorname{Rc}+(p+m)g-\nabla\nabla\ln u|^2ud \mu_{g(t)}  \notag \\
& -2\int_M| \operatorname{Rc}+(p+m)g|^2ud \mu_{g(t)} +2m(m+1)\int_M pud \mu_{g(t)}  \notag \\
& -2m\int_M \Delta p\cdot ud \mu_{g(t)}+ 2 \int_M \Delta\ln u\cdot pud \mu_{g(t)}. \label{eq 9 old}
\end{align}

For the further simplification we rewrite the following three expressions in (\ref{eq 9 old}).
Using scalar curvature $R(g(t))\equiv-m(m+1)$ and CRF (\ref{eq crf 1}) we have
\begin{align}
& | \operatorname{Rc}+(p+m)g-\nabla\nabla\ln u|^2   \notag \\
=& | \operatorname{Rc}+mg-\nabla\nabla\ln u|^2-2p\Delta\ln u+(m+1)p^2, \label{eq split pg I out} \\
& | \operatorname{Rc}+(p+m)g|^2=| \operatorname{Rc}+mg|^2+(m+1)p^2,
  \label{eq split out pg II}  \\
& \Delta p=(m+1)p-\frac{1}{m}|\operatorname{Rc}+mg|^2.  \label{eq 12 old}
\end{align}
Plugging equations (\ref{eq split pg I out}), (\ref{eq split out pg II}),  and (\ref{eq 12 old})
 into (\ref{eq 9 old}) and simplifying we get
\begin{align}
& \frac{d}{dt}\int_{M} \left (|\nabla\ln u|^2+2m\ln u \right )ud \mu_{g(t)} \notag \\
= & 2\int_M| \operatorname{Rc}+mg-\nabla\nabla\ln u|^2ud \mu_{g(t)}
-2\int_M\Delta\ln u\cdot pud \mu_{g(t)}.
\label{eq 11 old}
\end{align}

To simplify the last term in (\ref{eq 11 old}), using integration by parts and equation
(\ref{eq 12 old}) we have
\begin{align}
\int_M\Delta\ln u \cdot pu d \mu_{g(t)}
=&-\int_M |\nabla\ln u|^2pud \mu_{g(t)}+\int_M\Delta p\cdot ud \mu_{g(t)}  \notag \\
= & -\int_M |\nabla\ln u|^2pud\mu_{g(t)}+(m+1)\int_M pud \mu_{g(t)}  \notag \\
&      -\frac{1}{m}\int_M| \operatorname{Rc}+mg|^2ud \mu_{g(t)}   \notag \\
= & -\int_M |\nabla\ln u|^2pud\mu_{g(t)}-\int_M |\nabla\ln u|^2ud\mu_{g(t)} \notag \\
& +\frac{d \tilde{\mathcal{E}}(t)}{dt}  -\frac{1}{m}\int_M| \operatorname{Rc}
+mg|^2ud\mu_{g(t)},
\label{eq 13 old}
\end{align}
where we have used equation (\ref{eq 3 old}) to get the last equality.
Now we plug equation (\ref{eq 13 old}) into (\ref{eq 11 old}) and obtain
\begin{align*}
& \frac{d}{dt}\int_{M}\big(|\nabla\ln u|^2 + 2m\ln u\big)ud \mu_{g(t)}\\
=& 2\int_M |\operatorname{Rc} +mg-\nabla\nabla\ln u|^2ud \mu_{g(t)}
+2\int_M|\nabla\ln u|^2pud \mu_{g(t)} \\
& +2\int_M|\nabla \ln u|^2ud \mu_{g(t)} -2 \frac{d \tilde{\mathcal{E}}(t)}{dt}
 +\frac{2}{m}\int_M| \operatorname{Rc} +mg|^2u
d \mu_{g(t)}.
\end{align*}
This proves formula (\ref{eq 14 old}).

Since  pressure function $p \geq 0$, we have $\frac{d \tilde{\mathcal{W}}(t)}{dt} \geq 0$.
When $\frac{d \tilde{\mathcal{W}}(t)}{dt}=0$,  then each term on the
right-hand-side of equation (\ref{eq 14 old}) equals to zero. In particular,
$\operatorname{Rc}=-mg$.
\hfill $\square$

\begin{remark} Although $\partial_t p$ comes up in the proof of the
monotonicity of entropy $\tilde{\mathcal{W}}$,  $\tilde{\mathcal{W}}$ itself do not contain
any $p$-term. One way to think about it is that this is about the monotonicity formula for
the following flow
\[
 \partial_t g=-2\left ( \operatorname{Rc}_{g(t)}+ \left (
 \frac{1}{m}\left ( -\Delta_{g(t)}+(m+1) \right )^{-1}
 |\operatorname{Rc}_{g(t)}+mg(t)|^2 +m \right )g(t)\right ).
\]
\end{remark}

\section{Another proof of Theorem \ref{thm funct W monotone} and some remarks}
\label{sec second proof}

Following the idea of \cite[\S 9]{Per02} now we give another proof of
Theorem \ref{thm funct W monotone}.
Let $u=e^{-f}$ be a positive solution of (\ref{eq 2 old}) on closed manifold $M^{m+1}$,
then
\[
\partial_t f + \Delta f -| \nabla f|^2  +(m+1)p= 0.
\]
Note that we can rewrite the entropy
\begin{align*}
\tilde{\mathcal{W}}(g,e^{-f})& =\int_{M}\left ( |\nabla f|^2-2(m+1) f \right )e^{-f}d{\mu_{g}} \\
& =\int_{M}\left ( 2 \Delta f - |\nabla f|^2-2(m+1) f \right )e^{-f}d{\mu_{g}} .
\end{align*}
We will compute the evolution equation of
\begin{equation}
v= \left (2 \Delta f -| \nabla f|^2 -2(m+1) f \right  )e^{-f} \label{eq def v in Harnack}
\end{equation}
using backward parabolic operator $\square^* = -\partial_t -\Delta_{g(t)} +(m+1)p$.

Let
\[
w = 2 \Delta f -| \nabla f|^2 -2(m+1) f .
\]
 We compute
\begin{align*}
\square^* w =& -2 \left (\partial_t   + \Delta \right ) \Delta f
    + (\partial_t  + \Delta) | \nabla f|^2  +2(m+1) p  \Delta f  \\
& -(m+1) p | \nabla f|^2+ 2(m+1) \left (\partial_t + \Delta -(m+1)p \right ) f \\
= & -2 \Delta \left (\partial_t   + \Delta \right )  f -4R_{ij} \nabla_i \nabla_j f -4(p+m) \Delta f
-4 \nabla_i p \nabla_i f \\
& +  4R_{ij}\nabla_i f \nabla_j f + 2 |\nabla \nabla f|^2  +2(p+m)  | \nabla f|^2 + 2 \nabla_i
 |\nabla f|^2  \nabla_i f  \\
 &  +2(m+1) p  \Delta f  -(m+1) p | \nabla f|^2 \\
& + 2(m+1) \left (  | \nabla f|^2 - (m+1)p -(m+1)pf \right ),
\end{align*}
where to get the last equality we have used
\[
\partial_t  \Delta f =  \Delta \partial_t  f +2R_{ij} \nabla_i \nabla_j f +2(p+m) \Delta f
-(m-1) \nabla_i p \nabla_i f
\]
(see \cite[(S.5) on p.547]{CLN06}) and the parabolic Bochner formula
\begin{align*}
(\partial_t  + \Delta) | \nabla f|^2
=& 4R_{ij}\nabla_i f \nabla_j f + 2 |\nabla \nabla f|^2  +2(p+m)  | \nabla f|^2 \\
& + 2 \nabla_i  |\nabla f|^2  \nabla_i f -2(m+1) \nabla_i p
\nabla_i f.
\end{align*}

Using
\begin{align*}
 -2 \Delta \left (\partial_t   + \Delta \right )  f =&  -2  \Delta |\nabla f|^2 +2(m+1) \Delta p \\
= &- 4R_{ij}\nabla_i f \nabla_j f - 4 |\nabla \nabla f|^2 -4 \nabla_i  \Delta f \nabla_i f  \\
& +2(m+1)^2 p -\frac{2(m+1)}{m} |\operatorname{Rc} + mg|^2
\end{align*}
and by some simplification we get
\begin{align*}
\square^* w =& -2 | \operatorname{Rc} + mg + \nabla \nabla f|^2  - \frac{2}{m} |
 \operatorname{Rc} + mg|^2 -2 \nabla_i f \nabla_i w-2 | \nabla f|^2 \\
 & +(m-1) p w  -4 \nabla_i p \nabla_i f -4(m+1)pf.
\end{align*}

Now we compute
\begin{align}
& \square^* v =\square^* (wu) = \square^* (w e^{-f})  \notag \\
= & (\square^* w) u +w (\square^* u) -2 \nabla w \nabla u -  (m+1) pwu \notag \\
=&  -2 | \operatorname{Rc} + mg + \nabla \nabla f|^2 u  - \frac{2}{m} |
 \operatorname{Rc} + mg|^2 u  \notag \\
 & -2 | \nabla f|^2 u -4 \nabla_i p \nabla_i f u -4(m+1)pf u  -2pw u \notag \\
 =&  -2 | \operatorname{Rc} + mg + \nabla \nabla f|^2 u  - \frac{2}{m} |
 \operatorname{Rc} + mg|^2 u  \notag \\
 & -2 | \nabla f|^2 u -2 | \nabla f|^2 p u  +4 \operatorname{div}( p\nabla u).
 \label{eq ptwise box 3 star v}
\end{align}
The formula (\ref{eq 14 old}) follows from (\ref{eq ptwise box 3 star v}) easily,
we omit the detail.

\begin{remark}
Because of the term $ \operatorname{div}( p\nabla u)$ in formula (\ref{eq ptwise box 3 star v})
we are not able to prove a pointwise Harnack inequality for heat kernel function $u$
as given in \cite[Corollary 9.3]{Per02} and/or \cite[Theorem 1.2]{Ni04}.
\end{remark}

On a closed manifold $M^{m+1}$ we define functional
\begin{equation} \label{eq nu function tilde W}
\tilde{\nu} (g)  = \inf_{\int_M e^{-f} d \mu_g =1} \tilde{\mathcal{W}}(g,e^{-f})
\end{equation}
for each Riemannian metric $g$.
To see that $\tilde{\nu} (g)$ is a finite  number, recall that
the log-Sobolev inequality for $(M,g)$ says (\cite[\S 3]{Per02}, \cite[(1.8)]{Ni04})
that there is a constant $c_1 =c_1(M,g)$ such that
\begin{equation} \label{eq log sob const c1}
\int_M  \left (|\nabla f|^2 + f \right )e^{-f}d{\mu_{g}} \geq c_1
\end{equation}
for $f$ satisfying $\int_M e^{-f} d \mu_g =1$. Since $x e^{-x} \leq e^{-1}$ for any real $x$,
we have
\begin{align*}
\tilde{\mathcal{W}}(g,e^{-f}) = & \int_M  \left (|\nabla f|^2 + f \right )e^{-f}d{\mu_{g}}
-(2m+3) \int_M f e^{-f} d \mu_g \\
\geq & c_1 -(2m+3) e^{-1} \operatorname{vol}(M,g).
\end{align*}

An easy consequence the monotonicity formula (\ref{eq 14 old}) is

\begin{lemma}
$\tilde{\nu}(g(t))$ is monotone increasing under CRF (\ref{eq crf 1}).
\end{lemma}

\begin{proof}
This can be proved as the well-known proof in \cite[\S 2.2]{Per02}.
\end{proof}

\begin{remark}
It is not clear to us whether the lower bound $c_1$ in (\ref{eq log sob const c1})
depends on $c_2$ only assuming that $c_2$ is the lower bound in  $\tilde{\mathcal{W}}(g,e^{-f})
\geq c_2$ for $\int_M e^{-f} d \mu_g =1$. If it is true then we will have the volume
$\kappa$-noncollapsing on scale $r=1$ for CRF $g(t)$.
\end{remark}

\begin{remark} A naive generalization to CRF of the monotonicity of the reduced volume as
in \cite[Theorem 1.4]{Mul10} does not hold since the nonnegative condition \cite[(1.6)]{Mul10} is false
for CRF.
\end{remark}


\end{document}